\documentclass[11pt,reqno]{amsart}

\usepackage{amssymb,latexsym,amsmath,amsfonts}
\usepackage{amsmath,amscd,amsthm,amssymb}
\usepackage{hyperref}

\numberwithin{equation}{section}

\def\whitebox{{\hbox{\hskip 1pt
 \vrule height 6pt depth 1.5pt
 \lower 1.5pt\vbox to 7.5pt{\hrule width
    3.2pt\vfill\hrule width 3.2pt}%
 \vrule height 6pt depth 1.5pt
 \hskip 1pt } }}
\def\dmid{\,||\,}
\def\qed{\ifhmode\allowbreak\else\nobreak\fi\hfill\quad\nobreak
     \whitebox\medbreak}

\newcommand{\ignore}[1]{}

\newcommand{\ga}{\gamma}
\newcommand{\F}{\mathbb{F}}
\newcommand{\Z}{\mathbb{Z}}

\newcommand{\Q}{\mathbb{Q}}
\newcommand{\la}{\langle}
\newcommand{\ra}{\rangle}

\newcommand{\Cc}{{\mathbb C}}

\newcommand{\cP}{{\mathcal P}}

\newcommand{\Trace}{{\rm Tr}}
\newcommand{\Aut}{{\rm Aut}}
\newcommand{\Dev}{{\rm Dev}}
\newcommand{\Cay}{{\rm Cay}}

\newtheorem{thm}{Theorem}[section]
\newtheorem{lemma}[thm]{Lemma}

\newtheorem{remark}[thm]{Remark}

\newtheorem{example}[thm]{Example}
\numberwithin{equation}{section}

\begin{document}

\title[Skew Hadamard Difference Sets]
{Cyclotomic Constructions of Skew Hadamard Difference Sets}

\author[Feng and Xiang]{Tao Feng, Qing Xiang$^*$}

\thanks{$^*$Research supported in part by NSF Grant DMS 1001557, and by the
Overseas Cooperation Fund (grant 10928101) of China.}

\address{Department of Mathematical Sciences, University of Delaware, Newark, DE 19716, USA}
\email{feng@math.udel.edu}

\address{Department of Mathematical Sciences, University of Delaware, Newark, DE 19716, USA} \email{xiang@math.udel.edu}

\keywords{Cyclotomy, Difference set, Gauss sum, Index 2 Gauss sum, Partial difference set, Skew Hadamard difference set.}

\begin{abstract} 
We revisit the old idea of constructing difference sets from cyclotomic classes. Two constructions of skew Hadamard difference sets are given in the additive groups of finite fields by using union of cyclotomic classes of $\F_q$ of order $N=2p_1^m$, where $p_1$ is a prime and $m$ a positive integer. Our main tools are index 2 Gauss sums, instead of cyclotomic numbers. 
\end{abstract}

\maketitle

\section{Introduction}
We assume that the reader is familiar with the basic theory of difference sets as can be found in \cite{Lander} and Chapter 6 of \cite{bjl}. For a survey of recent progress in this area we refer the reader to \cite{Xiang}.

A difference set $D$ in a finite group $G$ is called {\it skew Hadamard} if $G$ is the disjoint union of $D$, $D^{(-1)}$, and $\{1\}$, where $D^{(-1)}=\{d^{-1}\mid d\in D\}$. The primary example (and for many years, the only known example in abelian groups) of skew Hadamard difference sets is the classical Paley difference set in $(\F_q,+)$ consisting of the nonzero squares of $\F_q$, where $\F_q$ is the finite field of order $q$, and $q$ is a prime power congruent to 3 modulo 4.  Skew Hadamard difference sets are currently under intensive study, see \cite{dy, dwx, wqwx, wh, tfeng, muzy}.  There were two major conjectures in this area:  (1) If  an abelian group $G$ contains a skew Hadamard difference set, then $G$ is necessarily elementary abelian. (2) Up to equivalence the Paley difference sets mentioned above are the only skew Hadamard difference sets in abelian groups. The first conjecture is still open in general. We refer the reader to \cite{cxs} for the known results on this conjecture. The second conjecture failed spectacularly: Ding and Yuan \cite{dy} constructed a family of skew Hadamard difference sets in $(\F_{3^m},+)$, where $m\geq 3$ is odd, and showed that the 2nd and the 3rd examples in the family  are inequivalent to the Paley difference sets. Very recently, Muzychuk \cite{muzy} constructed exponentially many inequivalent skew Hadamard difference sets over an elementary abelian group of order $q^3$.

We give a short survey of the known constructions of skew Hadamard difference sets. Shortly after the appearance of the Ding-Yuan construction \cite{dy}, by using certain permutation polynomials arising from the Ree-Tits slice symplectic spreads, Ding, Wang and Xiang \cite{dwx} constructed a class of skew Hadamard difference sets in $(\F_{3^m}, +)$, where $m$ is odd. Next the classical Paley construction was generalized from finite fields to commutative semifields in \cite{wqwx}. As a consequence, every finite commutative semifield of order congruent to 3 modulo 4 gives rise to a skew Hadamard difference set. The first author \cite{tfeng}  then constructed a family of skew Hadamard difference sets in the nonabelian group of order $p^3$ and exponent $p$, where $p$ is an odd prime. Prior to \cite{tfeng}, only two nonabelian skew Hadamard difference sets were known \cite{kibler}; both are in nonabelian groups of order $27$. Motivated by \cite{tfeng}, Muzychuk \cite{muzy} has now given a prolific construction of skew Hadamard difference sets in elementary abelian groups of order $q^3$, where $q$ is a prime power. We also mention that the construction in \cite{tfeng} was recently generalized in \cite{cpol} and \cite{chenfeng}.

Let $q=p^f$, where $p$ is a prime and $f$ a positive integer. Let $\gamma$ be a fixed primitive element of $\F_q$ and $N|(q-1)$ with $N>1$. Let $C_0=\langle \gamma^N\rangle$, and $C_i=\gamma^i C_0$ for $1\leq i\leq N-1$. The $C_i$ are called the {\it cyclotomic classes of order $N$} of $\F_q$. In this paper, we give two constructions of skew Hadamard difference sets in the additive groups of finite fields by using unions of cyclotomic classes. 

The idea of constructing difference sets (and strongly regular Cayley graphs) from cyclotomic classes of course goes back to  Paley \cite{paley}.  In the mid-20th century, Baumert, Chowla,  Hall, Lehmer, Storer, Whiteman, Yamamoto, etc. pursued this line of research vigorously. Storer's book \cite{storer} contains a summary of results in this direction up to 1967. See also Chapter 5 of \cite{baumert} for a summary. This method for constructing difference sets, however, has had only very limited success. Let $C_i$ be as above. It is known \cite[p.~123--124]{bjl} that a single cyclotomic class can form a difference set in $(\F_q,+)$ if $N=2,4,$ or $8$ and $q$ satisfies certain conditions. (Note that in order to obtain difference sets this way, the conditions on $q$ are quite restrictive when $N=4$ or $8$). It is conjectured that the converse is also true. Namely, if $C_0$ is a difference set in $(\F_q,+)$, then $N$ is necessarily $2, 4$, or $8$. This conjecture has been verified \cite{evans} up to $N=20$. If one uses a union of cyclotomic classes, instead of just one single class, the only new family of difference sets found in this way is the Hall sextic difference sets in $(\F_q,+)$ formed by taking a union of three cyclotomic classes of order 6, where $q=4x^2+27$ is a prime power congruent to $1$ modulo $6$. One of the reasons that very few difference sets have been discovered by using unions of cyclotomic classes is that the investigations often relied on the so-called cyclotomic numbers and these numbers are in general very difficult to compute if $N$ is large. 

In this paper, we construct skew Hadamard difference sets in $(\F_q,+)$ by using union of cyclotomic classes of order $N=2p_1^m$ of $\F_q$, where $p_1$ is an odd prime, $q$ is a power of a prime $p$, $\gcd(p,N)=1$, $-1\not\in \langle p\rangle\subset (\Z/N\Z)^*$, and the order of $p$ modulo $N$ is half of $\phi(N)$ (here $\phi$ is the Euler phi function). This last condition is the so-called index 2 condition in the theory of Gauss sums.  The significance of our constructions is twofold. First,  other than the classical Paley difference sets, previously known abelian skew Hadamard difference sets were constructed either in $(\F_q,+)$, where $q$ is a power of $3$, or in groups of order $p^{3k}$, where $p$ is an odd prime.  The constructions in this paper can produce skew Hadamard difference sets in elementary abelian groups where no previous constructions were known except for the classical Paley difference sets. Secondly our constructions demonstrate that the old idea of constructing difference sets from cyclotomic classes is not a dead end, thus it should be further exploited. 

The recent success in constructing strongly regular Cayley graphs by using union of cyclotomic classes in \cite{FX} lends further credence to our opinion on cyclotomic methods for constructing difference sets. Since the constructions in this paper are motivated by those in \cite{FX}, we include here a brief discussion 
of the results in \cite{FX}. Let $D$ be a subset of $\F_{q}$ such that $-D=D$ and $0 \not \in D$. We define the {\it Cayley graph} $\Cay(\F_{q}, D)$ to be the graph with the elements of $\F_{q}$ as vertices; two vertices are adjacent if and only if their difference belongs to $D$. When $D$ is a subgroup of the multiplicative group $\F_{q}^*$ of  $\F_{q}$ and  $\Cay(\F_{q}, D)$ is strongly regular,  then we speak of a {\it cyclotomic strongly regular graph}. In particular, if $D$ is the subgroup of $\F_{q}^*$ consisting of the squares, where $q$ is a prime power congruent to 1 modulo 4,  then $\Cay(\F_{q}, D)$ is the well-known Paley graph. In \cite{FX}, we were interested in examples due to De Lange \cite{del} and Ikuta and Munemasa \cite{ikutam},  in which one single cyclotomic class does not give rise to a strongly regular Cayley graph while a union of several classes does. Generalizing these examples, we give two constructions of strongly regular Cayley graphs on finite fields $\F_q$ by using union of cyclotomic classes of $\F_q$ of order $N$, where $N=p_1^m$ or $p_1^mp_2$, $p_1$ and $p_2$ are distinct odd primes. The main tools used in the proofs are index $2$ Gauss sums.  In particular, we \cite{FX} obtain twelve infinite families of strongly regular Cayley graphs with new parameters.

\section{Gauss sums}
Let $p$ be a prime, $f$ a positive integer, and $q=p^f$. Let $\xi_{p}=e^{2\pi i/p}$ and let $\psi$ be the additive character of $\F_{q}$ defined by 
\begin{equation}\label{defaddchar}
\psi: \F_{q} \rightarrow \Cc^{*}, \quad \psi(x)=\xi_{p}^{\Trace(x)},
\end{equation}
where ${\rm Tr}$ is the absolute trace from $\F_q$ to $\F_p$. Let $\chi:\F_{q}^* \rightarrow \Cc^{*} $ be a character of $\F_{q}^{*}$. We define the {\it Gauss sum} by
$$ g_{\atop{\F_q}}(\chi)=\sum_{a \in \F_{q}^*} \chi(a)\psi(a).$$
Usually we simply write $g(\chi)$ for $g_{\atop{\F_q}}(\chi)$ if the finite field involved is clear from the context. Note that if $\chi_0$ is the trivial multiplicative character of
$\F_q$, then $g(\chi_0)=-1$. We are usually concerned with nontrivial Gauss sums $g(\chi)$, i.e., those with $\chi\neq \chi_0$. Gauss sums can be viewed as the
Fourier coefficients in the Fourier expansion of $\psi|_{\F_q^*}$
in terms of the multiplicative characters of
 $\F_q$. That is, for every $c\in \F_q^*$,
\begin{equation}\label{inv}
\psi(c)=\frac{1} {q-1}\sum_{\chi\in {\widehat \F_q^*}}g({\bar \chi})\chi(c),
\end{equation}
where ${\bar \chi}=\chi^{-1}$ and ${\widehat \F_q^*}$ denotes the complex character group of $\F_q^*$.

We first recall a few elementary properties of Gauss sums. For proofs of these properties, see \cite[Theorem 1.1.4]{bew}. The first is
\begin{equation}\label{absvalue} 
g(\chi){\overline {g(\chi)}}=q, \; {\rm if}\; \chi\neq \chi_0.
\end{equation}
The second is
\begin{equation}\label{pthpower}
g(\chi^p)=g(\chi), 
\end{equation}
and the third one is
\begin{equation}\label{conjugate}
g(\chi^{-1})=\chi(-1){\overline {g(\chi)}}.
\end{equation}

While it is easy to determine the absolute value of nontrivial Gauss sums (see (\ref{absvalue})), the explicit evaluation of Gauss sums is a difficult problem. However, there are a few cases where the Gauss sums $g(\chi)$ can be explicitly evaluated.  The simplest case is the so-called {\it semi-primitive case}, where there exists an integer $j$ such that $p^j\equiv -1$ (mod $N$) (here $N$ is the order of $\chi$ in ${\widehat \F_q^*}$). Some authors \cite{BMW, bew} also refer to this case as uniform cyclotomy, or pure Gauss sums. We refer the reader to \cite[p.~364]{bew} for the precise evaluation of Gauss sums in this case. 

The next interesting case is the index 2 case, where $-1$ is not in the subgroup $\la p\ra$, the cyclic group generated by $p$,   and $\la p\ra$ has index 2 in $(\Z/N\Z)^*$ (again here $N$ is the order of $\chi$ in ${\widehat \F_q^*}$). Many authors have studied this case, including Baumert and Mykkeltveit \cite{bmy}, McEliece \cite{McE}, Langevin \cite{Lang}, Mbodj \cite{Mbo}, Meijer and Van der Vlugt \cite{mv}, and Yang and Xia  \cite{yx}. In the index 2 case, it can be shown that $N$ has at most two odd prime divisors. For the purpose of constructing difference sets by taking union of cyclotomic classes, we will need $N$ to be even and $(q-1)/N$ to be odd (cf. \cite[p.~357]{bjl}). Below is the result on evaluation of Gauss sums that we will need in Section 3.

\begin{thm}\label{gs}{\em (\cite[Theorem 4.4, Case D]{yx})}
 Let $N=2p_1^m$, where $p_1>3$ is a prime, $p_1\equiv 3\pmod{4}$ and $m$ is a positive integer. Assume that $p$ is a prime, $\gcd(p,N)=1$, $[(\Z/N\Z)^*:\langle p\rangle]=2$.  Let $f=\phi(N)/2$, $q=p^f$,  and $\chi$ be a character of order $N$ of $\F_q^*$.  Then
 \[ g(\chi)=\begin{cases}(-1)^{\frac{p-1}{2}m}\sqrt{p^*}p^{\frac{f-1}{2}},\quad\quad\quad\quad\quad\quad\quad\quad\;\textup{ if }p_1\equiv 7\pmod{8},\\
(-1)^{\frac{p-1}{2}(m+1)}\sqrt{p^*}p^{\frac{f-1}{2}-h}(\frac{b+c\sqrt{-p_1}}{2})^2,\quad \textup{ if }p_1\equiv 3\pmod{8},\end{cases}\]
where $h$ is the class number of $\Q(\sqrt{-p_1})$, $p^*=(-1)^{\frac{p-1}{2}}p$, and $b,c$ are integers satisfying the following conditions

(i) $b,c\not\equiv 0 \pmod p$,
(ii) $b^2+p_1c^2=4p^{h}$,
(iii) $bp^{\frac {f-h}{2}}\equiv -2 \pmod {p_1}$.
Moreover, when $p\equiv 3\pmod{4}$, we have
\begin{enumerate}
\item for $0\leq t\leq m-1$,  

\[g(\chi^{p_1^t})=\begin{cases}(-1)^{m}\sqrt{-p} p^{(f-1)/2},\quad\quad\quad\quad\quad\quad\quad\quad\quad\;\textup{ if }p_1\equiv 7\pmod{8},\\
  (-1)^{m+1}\sqrt{-p} p^{(f-1)/2-hp_1^t}(\frac{b+c\sqrt{-p_1}}{2})^{2p_1^t},\quad\; \textup{ if }p_1\equiv 3\pmod{8};\end{cases}\]
 
\item $g(\chi^{2p_1^t})=p^{(f-p_1^th)/2}(\frac{b+c\sqrt{-p_1}}{2})^{p_1^t}$;

\item $g(\chi^{p_1^m})=(-1)^{(f-1)/2}p^{(f-1)/2}\sqrt{-p}$.
\end{enumerate}
\end{thm}


We will also need the Stickelberger congruence for Gauss sums, which we state below.  Let $p$ be a prime, $q=p^f$, and let $\xi_{q-1}$ be a complex primitive $(q-1)$th root of unity. Fix any prime ideal $\mathfrak{P}$ in $\Z[\xi_{q-1}]$ lying over $p$. Then $\Z[\xi_{q-1}]/\mathfrak{P}$ is a finite field of order $q$, which we 
identify with $\F_q$. Let $\omega_{\mathfrak{P}}$ be the 
Teichm\"uller character on $\F_q$, i.e., an isomorphism
$$\omega_{\mathfrak{P}}: \F_q^{*}\rightarrow
\{1,\xi_{q-1},\xi_{q-1}^2,\dots ,\xi_{q-1}^{q-2}\}$$ 
satisfying 
\begin{equation}\label{eq2.3}
\omega_{\mathfrak{P}}(\alpha)\quad ({\rm
mod}\hspace{0.1in}{\mathfrak{P}})=\alpha,
\end{equation}  
for all $\alpha$ in $\F_q^*$.
The Teichm\"uller character $\omega_{\mathfrak{P}}$ has order $q-1$;
hence it generates all multiplicative characters of $\F_q$. 

Let $\cP$ be the prime ideal of $\Z[\xi_{q-1},\xi_p]$ lying above
$\mathfrak{P}$. For an integer $a$, let
$$s(a)=\nu _{\cP}(g(\omega_{\mathfrak P}^{-a})),$$ 
where $\nu_{\cP}$ is the $\cP$-adic valuation. 
Thus ${\cP}^{s(a)}\dmid g(\omega_{\mathfrak P}^{-a})$. 
The following evaluation of
$s(a)$ is due to Stickelberger (see \cite[p.~344]{bew}).

 \begin{thm}\label{stick} 
Let $p$ be a prime and $q=p^f$.
For an integer $a$ not divisible by $q-1$,
let $a_0+a_1p+a_2p^2+\cdots +a_{f-1}p^{f-1}$, $0\leq a_i\leq p-1$,
be the $p$-adic expansion of the reduction of $a$ modulo $q-1$.
Then
$$s(a)=a_0+a_1+\cdots +a_{f-1},$$
that is, $s(a)$ is the
sum of the $p$-adic digits of the reduction of $a$ modulo $q-1$.
Furthermore, define
$$t(a)=a_0!a_1!\cdots a_{f-1}!,\; \pi=\xi_p-1.$$
Then with $s(a)$ and $\omega_{\mathfrak P}$ as above we have the congruence
$$g(\omega_{\mathfrak P}^{-a})\equiv -
\frac{\pi^{s(a)}}{t(a)}
\quad (\mathrm{mod}\,{\cP}^{s(a)+1}).$$
\end{thm}

\section{Cyclotomic Constructions of Difference Sets}

We first recall a well-known lemma in the theory of difference sets.

\begin{lemma}\label{lem1}
Let $G$ be an abelian group of order $v$, $D$ be a subset of $G$ of size $k$, and let $\lambda$ be a positive integer. Then $D$ is a
$(v,k,\lambda)$ difference set in $G$ if and only if
$$\chi(D)\overline{\chi(D)}= {k-\lambda}$$
for every nontrivial complex character $\chi$ of $G$. Here, $\chi(D)$ stands for $\sum _{d\in D} ^{}\chi(d)$.  Moreover suppose that $D\cap D^{(-1)}=\emptyset$ and $1\not\in D$. Then $D$ is a skew Hadamard difference set in $G$ if and only if
\begin{equation}\label{skew}
\chi(D)=\frac {-1\pm \sqrt {-v}}{2}
\end{equation}
for every nontrivial complex character $\chi$ of $G$. 
\end{lemma}

Paley type partial difference sets are counterparts of skew Hadamard difference sets. We give the definition of these sets here.  Let $G$ be a finite (multiplicative) group of order $v$. A $k$-element subset $D$ of $G$ is called a {\em $(v, k, \lambda, \mu)$ partial difference set} (PDS, in short) provided that the list of ``differences'' $xy^{-1}$, $x, y \in D$, $x \neq y$, contains each nonidentity element of $D$ exactly $\lambda$ times and each nonidentity element of $G \backslash D$ exactly $\mu$ times. Furthermore, assume that $v\equiv 1$ (mod 4). A subset $D$ of $G$, $1\not\in D$, is called a {\em Paley type PDS} if $D$ is a $(v,\frac {v-1}{2}, \frac{v-5}{4}, \frac{v-1}{4})$ PDS. The set of nonzero squares in $\F_q$, $q\equiv 1$ (mod 4), is an example of Paley type PDS, which is usually called {\it the Paley PDS} in $\F_q$. The strongly regular Cayley graph constructed from the Paley PDS is the Paley graph. 

All constructions in this section are done in the following specific index 2 case: $N=2p_1^m$, $p_1>3$ is a prime, and $p_1\equiv 3$ (mod $4$); $p$ is a prime such that $\gcd(p,N)=1$, $-1\not\in \langle p\rangle\subset (\Z/N\Z)^*$, and $[(\Z/N\Z)^*:\langle p\rangle]=2$ (that is, $f:={\rm ord}_N(p)=\phi(N)/2$).

\subsection{The $p_1\equiv 7$ (mod $8$) case} We first give a construction of skew Hadamard difference sets in the case where $p_1\equiv 7$ (mod $8$). Let $p$ be a prime such that $\gcd(p,N)=1$. Write $f:={\rm ord}_N(p)=\phi(N)/2$ (so $N|(p^f-1)$). Let $E=\F_{q^s}$ be an extension field of $\F_q$, where $q=p^f$. Let $\gamma$ be a fixed primitive element of $E$, let $C_0=\langle \gamma^N\rangle$ and $C_i=\gamma^iC_0$ for $1\leq i\leq N-1$.

\begin{thm}\label{7mod8}
Assume that we are in the index 2 case as specified above, and $E=\F_{q^s}$ with $s$ odd. Let $I$ be any subset of $\Z/N\Z$ such that $\{i\pmod{p_1^m}\mid i \in I\}=\Z/p_1^m\Z$,  and let $D=\cup_{i\in I}C_i$. Then $D$ is a skew Hadamard difference set in $(E,+)$ if $p\equiv 3\pmod{4}$ and $D$ is a Paley type PDS if $p\equiv 1\pmod{4}$.
\end{thm}

\begin{proof} We shall only give the proof in the case where $p\equiv 3$ (mod $4$). The proof in the case where $p\equiv 1$ (mod $4$) is similar. First, we note that since $p\equiv 3$ (mod $4$) and $s$ is odd, we have $-1\in C_{p_1^m}$. By the choice of $I$, we have $-D\cap D=\emptyset$. Secondly, observe that since $p_1\equiv 7\pmod{8}$, we have $f-1=\frac{p_1-1}{2}p_1^{m-1}-1\equiv (-1)^m-1\pmod{4}$. Therefore $(-1)^{(f-1)/2}=(-1)^m$. Thirdly, let $\eta$ be any character of $\F_q^*$ of order $N$. By Theorem~\ref{gs} and the second observation, we have that for every $0\leq t\leq m$, $g_{_{\F_q}}(\eta^{p_1^t})=(-1)^{m}p^{(f-1)/2}\sqrt{-p}$; so by (\ref{conjugate}), 
$$g_{_{\F_q}}(\eta^{-p_1^t})=\eta^{p_1^t}(-1)\overline{g_{_{\F_q}}(\eta^{p_1^t})}=(-1)^{m}p^{(f-1)/2}\sqrt{-p}.$$ Now by the index 2 assumption, any integer in the set $\{i\mid 1\leq i\leq N-1, \gcd(i,N)=1\}$ is congruent (modulo $N$) to an element in $\langle p\rangle$ or an element in $-\langle p\rangle$. Therefore all odd integers in the interval $[1, N-1]$ are congruent to $\pm p^rp_1^t$ modulo $N$. Using (\ref{pthpower}) and the third observation above, with $\eta$ being any character of $\F_q^*$ of order $N$, we have
$$g_{\atop{\F_q}}(\eta^{-j})=(-1)^{m}p^{(f-1)/2}\sqrt{-p}$$
for all odd integers $j$, $1\leq j\leq N-1$. 

Now let $\chi$ be an arbitrary character of $E^*$ of order $N$. Since $N|(q-1)$, $\chi$ is the lift of some character $\eta$ of $\F_q^*$  and $o(\eta)=N$ (see \cite[Theorem 11.4.4]{bew}). Following the notation of \cite{bew}, we write $\chi=\eta'$. It follows that $\chi^j=(\eta^j)'$ for all $1\leq j\leq N-1$. Using the Davenport-Hasse theorem on lifted Gauss sums (\cite[p.~360]{bew}), we have that for all odd integers $j$, $1\leq j\leq N-1$,
\begin{equation}\label{odd}
g_{_{E}}(\chi^{-j})=(-1)^{s-1}g_{\atop{\F_q}}(\eta^{-j})^s=(-1)^mp^{\frac{s(f-1)}{2}}(\sqrt{-p})^s.
\end{equation}

We will prove the result stated in the theorem by using the second part of Lemma~\ref{lem1}. To this end, let $a$ be an arbitrary integer such that $0\leq a\leq N-1$ and let $\psi$ be the additive character of $E$ defined as in (\ref{defaddchar}), with $\F_q$ replaced by $E$. We compute 

\begin{align*}
\psi(\ga^a D)&=\sum_{i\in I}\psi(\ga^a C_i)\\
                    &=\frac{1}{N}\sum_{i\in I}\sum_{x\in E^*}\psi(\ga^{a+i}x^N)\\
                    &=\frac{1}{N}T_a,
\end{align*}
where
\[T_a=\sum_{\theta\in C_0^{\perp}}g_{_{E}}({\bar \theta})\sum_{i\in I}\theta(\ga^{a+i}).\]
Here $C_0^{\perp}$ is the unique subgroup of order $N$ of ${\widehat E^*}$. Note that in the last step of the above calculations we have used (\ref{inv}).

We now proceed to computing the sum $T_a$. If $\theta\in C_0^{\perp}$ and  $o(\theta)=1$, then $g_{_E}({\bar \theta})=-1$, and $\sum_{i\in I}\theta(\ga^{a+i})=p_1^m$. If $\theta\in C_0^{\perp}$, $o(\theta)\neq 1$, and $o(\theta)$ is odd, then  $\sum_{i\in I}\theta(\ga^{a+i})=\sum_{i=0}^{p_1^m-1}\theta(\ga^{a+i})=\theta(\ga^a)\frac{\theta(\ga)^{p_1^m}-1}{\theta(\ga)-1}=0$. Therefore, with $\chi$ a fixed generator of $C_0^{\perp}$, we have
$$T_a=-p_1^m+\sum_{j\, {\rm odd},\; 1\leq j\leq N-1}g_{_{E}}(\chi^{-j})\sum_{i\in I}\chi^j(\ga^{a+i}).$$
Using (\ref{odd}) and writing $\xi_{N}$ for $\chi(\ga)$, a complex primitive $N$th root of unity, we have
\begin{align*}
T_a=&-p_1^m+(-1)^m p^{s(f-1)/2}(\sqrt{-p})^s\sum_{u=0}^{p_1^m-1}\sum_{i\in I}\chi^{1+2u}(\ga^{a+i})\\
=&-p_1^m+(-1)^m p^{s(f-1)/2}(\sqrt{-p})^s\sum_{i\in I}\xi_N^{a+i}\left(\sum_{u=0}^{p_1^m-1}\xi_{p_1^m}^{u(a+i)}\right)
\end{align*}
For each $a$, $0\leq a\leq N-1$, there is a unique $i_a\in I$ such that $p_1^m|(a+i_a)$.  Write $a+i_a=p_1^mj_a$ for some integer $j_a$. Then 
$$T_a=-p_1^m+(-1)^{m}p^{s(f-1)/2}(\sqrt{-p})^s(-1)^{j_a}p_1^m.$$ 
It follows that
$$\psi(\ga^a D)=\frac{-1+(-1)^{m+j_a}\sqrt{-p^{sf}}}{2}.$$
By Lemma~\ref{lem1}, $D$ is a skew Hadamard difference set in $(E,+)$. The proof is now complete.
\end{proof}

\begin{example}\label{7and11} {\em Let $p_1=7$, $N=14$, $p=11$.  Then it is routine to check that ${\rm ord}_N(p)=3=\phi(N)/2$. Let $C_i$, $0\leq i\leq 13$, be the cyclotomic classes of order $14$ of $\F_{11^3}$.  

(1) Take $I=\{0,1,\ldots,6\}$. Then by Theorem~\ref{7mod8},  $D=C_0\cup C_1\cup \cdots \cup C_6$ is a skew Hadamard difference set in $(\F_{11^3},+)$.  Let $\Dev(D)$ denote the symmetric design developed from the difference set $D$. One can use a computer to find that $\Aut(\Dev(D))$ has size $5\cdot 11^3\cdot 19$. 

(2) Take $I=\{ 0, 1, 3, 4, 5, 6, 9\}$. Then by Theorem~\ref{7mod8}, $D'=C_0\cup C_1\cup C_3\cup C_4\cup C_5\cup C_6\cup C_9$ is also a skew Hadamard difference set in $(\F_{11^3},+)$. One finds by using a computer that $\Aut(\Dev(D'))$ has size $3\cdot5\cdot 11^3\cdot 19$. 

The automorphism group of the Paley design has size $3\cdot 5\cdot 7\cdot 11^3\cdot 19$. So the three difference sets $D$, $D'$ and the Paley difference set in $(\F_{11^3},+)$ are pairwise inequivalent. Also note that the sizes of the Sylow $p$-subgroups of $\Aut(\Dev(D))$ and $\Aut(\Dev(D'))$ are $q=11^3$,  while the size of the Sylow $p$-subgroups of the automorphism groups of the designs developed from the difference sets constructed by Muzychuk \cite{muzy} is strictly greater than $q$, we conclude that both $D$ and $D'$ are inequivalent to the corresponding skew Hadamard difference sets in \cite{muzy}.}
\end{example}

\begin{remark}
{\em (1) The automorphism group of the Paley design is determined in \cite{kantor}. Our construction in Theorem~\ref{7mod8} includes the Paley construction as a special case. It seems difficult to generalize the method in \cite{kantor} to determine the automorphism groups of the designs developed from our difference sets. Based on some computational evidence, we conjecture that $\Aut(\Dev(D))$, with $D=\cup_{i\in I}C_i$ as given in the statement of the theorem, is generated by the following three types of elements: (i) translations by elements of $\F_q$, (ii) multiplications by elements in $C_0$, and (iii) $\sigma_p^i$, $p^iI=I$, where $\sigma_p$ is the Frobenius automorphism of the finite field $\F_{p^{sf}}$.

(2) Let $N=2\cdot 7^m$, where $m\geq 2$. One can use induction to prove that ${\rm ord}_N(11)=3\cdot 7^{m-1}=\phi(N)/2$. Therefore the conditions of Theorem~\ref{7mod8} are satisfied. So the examples in Example~\ref{7and11} can be generalized into infinite families.}
\end{remark}

\subsection{The $p_1\equiv 3\pmod{8}$ case} We will again do the constructions in the index 2 case as specified at the beginning of this section. In addition, we will assume that
\begin{enumerate}

\item  $p_1\equiv 3\pmod{8}$, ($p_1\neq 3$),

\item  $N=2p_1$,

\item  $1+p_1=4p^h$, where $h$ is the class number of $\Q(\sqrt{-p_1})$,

\item  $p\equiv 3\pmod{4}$.

\end{enumerate}
Let $q=p^f$, $f={\rm ord}_N(p)=\phi(N)/2$. As in Section 2, let $\xi_{q-1}$ be a primitive complex $(q-1)$th root of unity, and ${\mathfrak P}$ be a prime ideal in $\Z[\xi_{q-1}]$ lying over $p$. Then $\Z[\xi_{q-1}]/\mathfrak{P}$ is a finite field of order $q$. We will use $\Z[\xi_{q-1}]/\mathfrak{P}$ as a model for $\F_q$. That is, 
\begin{equation}\label{model}
\F_q=\{{\overline 0}, {\overline 1}, {\overline \xi_{q-1}},  {\overline \xi_{q-1}^2}, \ldots , {\overline \xi_{q-1}^{q-2}}\}, 
\end{equation}
where ${\overline \xi_{q-1}}=\xi_{q-1} \pmod{{\mathfrak P}}$. Hence $\gamma:={\overline \xi_{q-1}}$ is a primitive element of $\F_q$. Let $\omega_{\mathfrak P}$ be the Teichm\"uller character of $\F_q$. Then 
$$\omega_{\mathfrak P}(\gamma)=\xi_{q-1}.$$
Let $\chi=\omega_{\mathfrak P}^{(q-1)/N}$. Then $\chi$ is a character of $\F_q^*$ of order $N=2p_1$ (and $\chi$ depends on the choice of ${\mathfrak P}$). To simplify notation, we write $\xi_N$ for $\chi(\gamma)=\xi_{q-1}^{(q-1)/N}$ and  $\xi_{p_1}$ for $\chi^2(\gamma)=\xi_{q-1}^{\frac{q-1}{p_1}}$. Next let $n=(q-1)/p_1$. By the result in \cite{Lang} (see also \cite[p.~376]{bew}), we have 

(1) $s(-n)=(p-1)b_0$, $s(n)=(p-1)b_1$ for some positive integers $b_0,b_1$, and $b_0>b_1=\frac{f-h}{2}$, 

(2) $g(\chi^2)=p^{\frac{f-h}{2}}\frac{(b+c\sqrt{-p_1})}{2}$, where $b, c$ are integers satisfying  (i) $b,c\not\equiv 0$ (mod $p$),
(ii) $b^2+p_1c^2=4p^{h}$,
(iii) $bp^{\frac {f-h}{2}}\equiv -2$ (mod $p_1$). 

By the assumption that $1+p_1=4p^h$, we must have $b,c\in \{1,-1\}$. The sign of $b$ is determined by the congruence in (iii).  The sign of $c$ depends on the choice of ${\mathfrak P}$. We have the following:

\begin{lemma}\label{signbc}
With notation as above, $bc\equiv -\sqrt{-p_1} \pmod {\mathfrak P}.$
\end{lemma}

\begin{proof} Let $\cP$ be the unique prime ideal of $\Z[\xi_{q-1},\xi_p]$ lying above ${\mathfrak P}$. Applying the Stickelberger congruence to $g(\chi^2)=g(\omega_{\mathfrak{P}}^n)$, we have
\begin{align*}
g(\chi^{2})=p^{\frac{f-h}{2}}\frac{(b+c\sqrt{-p_1})}{2}\equiv -\frac{\pi^{s(-n)}}{t(-n)}\pmod{\cP^{s(-n)+1}}.\\
\end{align*}
Now using the fact that $p=\prod_{i=1}^{p-1}(1-\xi_p^i)\equiv \pi^{p-1}\prod_{i=1}^{p-1}i\equiv -\pi^{p-1}\pmod{\cP^p}$, we can further simplify the above congruence to
\begin{align*}
b+c\sqrt{-p_1}&\equiv 2(-1)^{1+\frac {f-h}{2}}\frac{\pi^{(p-1)(b_0-(f-h)/2)}}{t(-n)}\equiv 0\pmod{\cP},\\
\end{align*}
where in the last step we have used the nontrivial fact that $b_0>b_1=\frac{f-h}{2}$. Therefore $bc\equiv -\sqrt{-p_1} \pmod \cP$. Since $\sqrt{-p_1}\in \Z[\xi_{p_1}]\subset \Z[\xi_{q-1}]$, we have $bc+\sqrt{-p_1}\in \cP\cap \Z[\xi_{q-1}]={\mathfrak P}$. That is, $bc\equiv -\sqrt{-p_1} \pmod {\mathfrak P}.$ The proof is complete. 
\end{proof}

Since $1+p_1=4p^h$, we have $(1+\sqrt{-p_1})(1-\sqrt{-p_1})\in {\mathfrak P}$ for any prime ideal ${\mathfrak P}$ of $\Z[\xi_{q-1}]$ lying above $p$. It follows that either $1+\sqrt{-p_1}\in {\mathfrak P}$ or $1-\sqrt{-p_1}\in {\mathfrak P}$. We can choose a prime ideal ${\mathfrak P}$ (and then fix this choice) such that 
\begin{equation}\label{choiceofprime}
1+\sqrt{-p_1}\in {\mathfrak P}.
\end{equation}
The corresponding $b,c$ in the evaluation of $g(\chi^2)$ will then satisfy $bc=1$ by Lemma~\ref{signbc}. These discussions were essentially done in \cite{yx2}. But there are a few minor problems in that paper. That is the reason why we gave the detailed account here.

By the index 2 assumption, we see that $\{i\pmod{p_1}\mid i\in\langle p\rangle\}$ is the set of nonzero squares of $\Z/p_1\Z$. Consequently, $\sum_{i\in\langle p\rangle}\xi_{p_1}^i=\frac{-1\pm \sqrt{-p_1}}{2}.$ It follows that $ 1+2\sum_{i\in\langle p\rangle}\xi_{p_1}^i\equiv \pm \sqrt{-p_1}\equiv \mp 1  \pmod {\mathfrak P}.$ Hence $1+2\sum_{i\in\langle p\rangle}\gamma^{in}=\mp 1$. Since $\sum_{i\in\langle p\rangle}\xi_{p_1}^i + \sum_{i\in\langle p\rangle}\xi_{p_1}^{-i}=-1$, we have 
$$(1+2\sum_{i\in\langle p\rangle}\gamma^{in})+(1+2\sum_{i\in\langle p\rangle}\gamma^{-in})=0.$$
So we can make a suitable choice of $\xi_{q-1}$ (that is, if necessary replace the originally chosen $\xi_{q-1}$ by $\xi_{q-1}^{-1}$) such that $1+2\sum_{i\in\langle p\rangle}\gamma^{in}=-1$.

We now give the construction of difference sets by using unions of cyclotomic classes. Let $\F_q$ be given as in (\ref{model}) with ${\mathfrak P}$ chosen in such a way that (\ref{choiceofprime}) holds, and $\gamma={\overline \xi_{q-1}}$ be the primitive element of $\F_q$ chosen above such that $1+2\sum_{i\in\langle p\rangle}\gamma^{in}=-1$. Let $N=2p_1$ and let $C_0=\langle \gamma^N\rangle$, and $C_i=\gamma^iC_0$ for $i=1,2, \ldots , N-1$,  be the cyclotomic classes of $\F_q$ of order $N$. Choose $I=\langle p\rangle\cup 2\langle p\rangle\cup\{0\}$, and define 
$$D:=\cup_{i\in I}C_i.$$ 

Note that $|I|=p_1$, and since $2$ is a quadratic nonresidue modulo $p_1$, we have $\{i\pmod{p_1}\mid  i\in I\}=\Z/p_1\Z$.

\begin{thm}\label{3mod8}
With the above definition,  $D$ is a skew Hadamard difference set in $(\F_q,+)$.
\end{thm}
\begin{proof} Since $-1\in C_{p_1}$ and $2$ is a quadratic nonresidue modulo $p_1$, we have $-D\cap D=\emptyset$. That is, $D$ is skew. 

We now proceed to proving the result by using the second part of Lemma~\ref{lem1}. To this end, let $a$ be an arbitrary integer such that $0\leq a\leq N-1$ and let $\psi$ be the additive character defined as in (\ref{defaddchar}). Also let $\chi=\omega_{\mathfrak P}^{(q-1)/N}$, where ${\mathfrak P}$ is chosen such that (\ref{choiceofprime}) holds. We have 
\begin{eqnarray*}
\psi(\ga^a D)&=&\sum_{i\in I}\psi(\ga^a C_i)\\
                    &=&\frac{1}{N}\sum_{i\in I}\sum_{x\in\F_q^*}\psi(\ga^{a+i}x^N)\\
                    &=&\frac{1}{N}T_a,
\end{eqnarray*}
where
\[T_a=\sum_{j=0}^{N-1}g(\chi^j)\sum_{i\in I}\chi^{-j}(\ga^{a+i}).\]
When $j=0$, we have $g(\chi^0)=-1$, and $\sum_{i\in I}\chi^{-j}(\ga^{(a+i)})=p_1$. If $j\neq 0$ is even, then  $\sum_{i\in I}\chi(\ga^{-j(a+i)})=\sum_{i=0}^{p_1-1}\chi(\ga^{-j(a+i)})=\chi(\ga^{-ja})\frac{\chi(\ga)^{-jp_1}-1}{\chi(\ga^{-j})-1}=0$. Now note that every odd integer in the interval 
$[1, N-1]$ is congruent (modulo $N$) to an element in $\langle p\rangle$, or an element in $-\langle p\rangle$, or $p_1$. 
Therefore, we have
\begin{align*}
T_a=&-p_1+\sum_{j\in \langle p\rangle}g(\chi^j)\sum_{i\in I}\bar{\chi}^j(\ga^{a+i})+\sum_{j\in -\langle p\rangle}g(\chi^j)\sum_{i\in I}\bar{\chi}^j(\ga^{a+i})\\
&+g(\chi^{p_1})\sum_{i\in I}\bar{\chi}^{p_1}(\ga^{a+i}).
\end{align*}

Specializing Theorem~\ref{gs} to the $m=1$ case and noting that $f-1\equiv 0$ (mod $4$) since $p_1\equiv 3$ (mod $8$), we have 
$$g(\chi)=p^{(f-1)/2-h}\sqrt{-p}\left(\frac{b+c\sqrt{-p_1}}{2}\right)^2,\; g(\chi^{p_1})=p^{(f-1)/2}\sqrt{-p},$$
where $b,c$ are the same as in the evaluation of $g(\chi^2)$ (c.f. \cite[p.~2531]{yx}), and $bc=1$ by our choice of ${\mathfrak P}$. Also recall that by (\ref{pthpower}), we have $g(\chi^p)=g(\chi)$. We compute 
\begin{align*}
T_a=&-p_1+p^{(f-1)/2-h}\sqrt{-p}\left(\frac{b+c\sqrt{-p_1}}{2}\right)^2\sum_{j\in \langle p\rangle}\sum_{i\in I}\bar{\chi}^j(\ga^{a+i})\\&+
p^{(f-1)/2-h}\sqrt{-p}\left(\frac{b-c\sqrt{-p_1}}{2}\right)^2\sum_{j\in -\langle p\rangle}\sum_{i\in I}\bar{\chi}^j(\ga^{a+i})\\&+
p^{(f-1)/2}\sqrt{-p}\sum_{i\in I}\bar{\chi}^{p_1}(\ga^{a+i}).
\end{align*}
    
Since $\{i\pmod{p_1}\mid i\in I\}=\Z/p_1\Z$, for each $a$, $0\leq a\leq N-1$, there is a unique $i_a$ in $I$ such that $p_1|(a+i_a)$. Write $a+i_a=p_1j_a$. We have 
\begin{align*}\label{easyterm}
&\sum_{j\in \langle p\rangle}\sum_{i\in I}\bar{\chi}^j(\ga^{a+i})+\sum_{j\in -\langle p\rangle}\sum_{i\in I}\bar{\chi}^j(\ga^{a+i})+\sum_{i\in I}\bar{\chi}^{p_1}(\ga^{a+i})\\
&=\sum_{j \,\textup{odd}}\xi_{N}^{-j(a+i)}=\sum_{i\in I}\xi_{2p_1}^{-(a+i)}\sum_{u=0}^{p_1-1}\xi_{p_1}^{-u(a+i)}\\
&=(-1)^{j_a}p_1.
\end{align*}
  
In order to prove that $D$ is a skew Hadamard difference set, we must show that 
\begin{equation}\label{condition}
T_a=-p_1+p^{(f-1)/2}\sqrt{-p}p_1\epsilon_a
\end{equation}
for some $\epsilon_a=\pm 1$. Noting that $(\frac{b+c\sqrt{-p_1}}{2})^2=\frac{b^2-c^2p_1+2bc\sqrt{-p_1}}{4}=\frac{1-p_1+2\sqrt{-p_1}}{4}$, and 
$$\sum_{j\in \langle p\rangle}\sum_{i\in I}\bar{\chi}^j(\ga^{a+i})+\sum_{j\in -\langle p\rangle}\sum_{i\in I}\bar{\chi}^j(\ga^{a+i})+\sum_{i\in I}\bar{\chi}^{p_1}(\ga^{a+i})=(-1)^{j_a}p_1,$$ one simplifies (\ref{condition}) to  
\[\frac{1-p_1}{4p^h}p_1(-1)^{j_a}+(1-\frac{1-p_1}{4p^h})(-1)^a+\frac{\sqrt{-p_1}}{2p^h}X_a=p_1\epsilon_a,\]
where $X_a=\sum_{j\in \langle p\rangle}\sum_{i\in I}\xi_N^{-j(a+i)}-\sum_{j\in -\langle p\rangle}\sum_{i\in I}\xi_N^{-j(a+i)}$.  Using the assumption that $1+p_1=4p^h$, one further simplifies the last equation to
\begin{align}\label{cdn}
(1-p_1)(-1)^{j_a}+2(-1)^a-2\frac{X_a}{\sqrt{-p_1}}=(1+p_1)\epsilon_a.
\end{align}
Below we will prove that (\ref{cdn}) always holds, thus proving that $D$ is a skew Hadamard difference set. We recall some useful facts.

(1) $\{i\pmod{p_1}\mid i\in\langle p\rangle\}$ is the set of nonzero squares of $\Z/p_1\Z$.

(2) The odd integers $\frac{p_1-1}{2}$ and $p_1-2$ are congruent to elements in $\langle p\rangle$ modulo $p_1$ since both $2$ and $-1$ are nonresidues modulo $p_1$.

(3) By our choice of $\xi_{q-1}$ we have $\sum_{i\in\langle p\rangle}\xi_{p_1}^i=\frac{-1+\sqrt{-p_1}}{2}$.

(4) We have $\sum_{i\in\langle p\rangle}\xi_{N}^{-i}=\sum_{i\in\langle p\rangle}\xi_{N}^{(2\cdot (p_1-1)/2-p_1)i}=\sum_{i\in\langle p\rangle}\xi_{p_1}^{\frac{p_1-1}{2}i}(-1)^i=-\frac{-1+\sqrt{-p_1}}{2}=\frac{1-\sqrt{-p_1}}{2}$, since $\frac{p_1-1}{2}$ is a square in $\Z/p_1\Z$.\\

Now set $Y_a:=\sum_{j\in \langle p\rangle}\sum_{i\in I}\xi_N^{-j(a+i)}$. Then $X_a=Y_a-\overline{Y_a}$. We have 
\[Y_a=\sum_{j\in \langle p\rangle}\xi_N^{-aj}\sum_{i\in I}\xi_N^{-ij}=(\sum_{j\in \langle p\rangle}\xi_N^{-aj})(\sum_{i\in I}\xi_N^{-i}).\] 
The last sum above,  $\sum_{i\in I}\xi_N^{-i}$, can be evaluated as follows: $\sum_{i\in I}\xi_N^{-i}=\sum_{i\in \langle p\rangle}\xi_N^{-i}+\sum_{i\in 2\langle p\rangle}\xi_N^{-i}+1=\frac{1-\sqrt{-p_1}}{2}+\frac{-1-\sqrt{-p_1}}{2}+1=1-\sqrt{-p_1}$. Therefore
$$Y_a=(\sum_{j\in \langle p\rangle}\xi_N^{-aj})(1-\sqrt{-p_1}).$$

We consider the following six cases.

\noindent {\bf Case 1.} $a=0$. In this case, $i_a=0$, $j_a=0$. We have $\sum_{j\in \langle p\rangle}\xi_N^{-aj}=\frac{p_1-1}{2}$, $Y_a=\frac{p_1-1}{2}(1-\sqrt{-p_1})$, and $X_a=-(p_1-1)\sqrt{-p_1}$. Condition (\ref{cdn}) is satisfied with $\epsilon_a=1$.\\

\noindent {\bf Case 2.}  $a\in\langle p\rangle$. In this case, $i_a=2\cdot\frac{p_1-1}{2}a\in 2\langle p\rangle$. $j_a=a\equiv 1 \pmod{2}$. We have $\sum_{j\in \langle p\rangle}\xi_N^{-aj}=\sum_{j\in \langle p\rangle}\xi_N^{-j}=\frac{1-\sqrt{-p_1}}{2}$, $Y_a=\frac{1-\sqrt{-p_1}}{2}\cdot(1-\sqrt{-p_1})=\frac{1-p_1}{2}-\sqrt{-p_1}$, and $X_a=-2\sqrt{-p_1}$. Condition (\ref{cdn}) is satisfied with $\epsilon_a=1$.\\

\noindent{\bf Case 3.} $a\in-\langle p\rangle$. In this case, $i_a=-a\in \langle p\rangle$, $j_a=0$. We have $\sum_{j\in \langle p\rangle}\xi_N^{-aj}=\sum_{j\in \langle p\rangle}\xi_N^{j}=\frac{1+\sqrt{-p_1}}{2}$, $Y_a=\frac{1+\sqrt{-p_1}}{2}\cdot(1-\sqrt{-p_1})=\frac{p_1+1}{2}$, and $X_a=0$. Condition (\ref{cdn}) is satisfied with $\epsilon_a=-1$.\\

\noindent{\bf Case 4.} $a\in 2\langle p\rangle$. In this case, write $a=2u$ for some $u\in\langle p\rangle$. We have $i_a=(p_1-2)u\in \langle p\rangle$, $j_a=u\equiv 1 \pmod{2}$, and $\sum_{j\in \langle p\rangle}\xi_N^{-aj}=\sum_{j\in \langle p\rangle}\xi_{p_1}^{-j}=\frac{-1-\sqrt{-p_1}}{2}$. It follows that $Y_a=\frac{-1-\sqrt{-p_1}}{2}\cdot(1-\sqrt{-p_1})=-\frac{p_1+1}{2}$. Thus $X_a=0$. Condition (\ref{cdn}) is satisfied with $\epsilon_a=1$.\\

\noindent{\bf Case 5.} $a\in -2\langle p\rangle$. In this case, $i_a=-a\in 2\langle p\rangle$, $j_a=0$. We have $\sum_{j\in \langle p\rangle}\xi_N^{-aj}=\sum_{j\in \langle p\rangle}\xi_{p_1}^{j}=\frac{-1+\sqrt{-p_1}}{2}$, $Y_a=\frac{-1+\sqrt{-p_1}}{2}\cdot(1-\sqrt{-p_1})=\frac{-1+p_1}{2}+\sqrt{-p_1}$, and $X_a=2\sqrt{-p_1}$. Condition (\ref{cdn}) is satisfied with $\epsilon_a=-1$.\\

\noindent{\bf Case 6.} $a=p_1$. In this case, $i_a=0$, $j_a=1$. We have $\sum_{j\in \langle p\rangle}\xi_N^{-aj}=-\frac{(p_1-1)}{2}$, $Y_a=-\frac{(p_1-1)}{2}(1-\sqrt{-p_1})$, and $X_a=(p_1-1)\sqrt{-p_1}$. Condition (\ref{cdn}) is satisfied with $\epsilon_a=-1$.

The proof is now complete.
\end{proof}

\begin{remark}
{\em It would be interesting to extend the construction in Theorem~\ref{3mod8} to the general case where $N=2p_1^m$, $p_1\equiv 3\pmod 8$ and $m\geq 2$ is arbitrary.  If this can be done, then we will obtain infinite families of skew Hadamard difference sets in this way even though currently we only know finitely many pairs $(p_1,p)$ such that $1+p_1=4p^h$, where $h$ is the class number of $\Q(\sqrt{-p_1})$.}
\end{remark}

\begin{example}
{\em Let $p=3, N=22,  p_1=11$. It is routine to check that ${\rm ord}_{22}(3)=5=\phi(N)/2$. Let $f=5$, $q=3^5$ and $n=\frac{q-1}{p_1}$. The class number $h$ of $\Q(\sqrt{-11})$ is 1 (c.f. \cite[p.~514]{hcohen}). Therefore the condition $1+p_1=4p^h$ is indeed satisfied. Let $\F_{3^5}$ be the finite field as in (\ref{model}) with ${\mathfrak P}$ chosen in such a way that (\ref{choiceofprime}) holds. Choose a primitive element $\gamma$ of $\F_{3^5}$ such that $\gamma$ such that $1+2\sum_{i\in\langle p\rangle}\gamma^{in}=-1$, and let $C_i$, $0\leq i\leq 21$, be the cyclotomic classes with respect to this choice of $\gamma$. Define $I:=\langle 3\rangle\cup 2\langle 3\rangle\cup \{0\}=\{0,1,2,3,5,6,8,9,10,15,18\}$. Then 
$D=\cup_{i\in I}C_i$ is a skew Hadamard difference set in $(\F_{3^5},+)$.  Using a computer one finds that $\Aut(\Dev(D))$ has size $3^5\cdot5\cdot 11$, while the automorphism group of the corresponding Paley design has size $3^5\cdot5\cdot 11^2$. We conclude that ${\rm Dev}(D)$ is not isomorphic to the Paley design.}
\end{example}


\begin{thebibliography}{99}

\bibitem{baumert} L. D. Baumert, {\it Cyclic Difference Sets}, Lecture Notes in Mathematics 182, Springer-Verlag, 1971.

\bibitem{bmy} L. D. Baumert, J. Mykkeltveit, Weight distributions of some irreducible cyclic codes, {\it DSN Progr. Rep.}, {\bf 16} (1973), 128--131.

\bibitem{BMW}L. D. Baumert, M. H. Mills, and R. L. Ward,  Uniform Cyclotomy, {\it J. Number Theory} {\bf 14} (1982), 67--82.

\bibitem{bew}B. C. Berndt, R. J. Evans, and K. S. Williams, {\it Gauss and Jacobi
Sums}, A Wiley-Interscience Publication, 1998.

\bibitem{bjl} T. Beth, D. Jungnickel, and H. Lenz, {\it Design Theory}. Vol. I. Second edition. Encyclopedia of Mathematics and its Applications, 78. Cambridge University Press, Cambridge, 1999.



\bibitem{cxs} Y. Q. Chen, Q. Xiang, and S. Sehgal, An exponent bound on skew Hadamard abelian difference sets,  {\it Designs, Codes and Cryptogr.}  {\bf 4}  (1994), 313--317.

\bibitem{cpol} Y. Q. Chen, J. Polhill, Paley type group schemes and planar Dembowski-Ostrom polynomials, {\it Disc. Math.} {\bf 311} (2011), 1349--1364.

\bibitem{chenfeng} Y. Q. Chen, T. Feng, Abelian and non-abelian Paley type group schemes, preprint.

\bibitem{hcohen} H. Cohen, {\it A course in computational algebraic number theory}, GTM 138, Springer, 1996.

\bibitem{dy} C. Ding, J. Yuan,  A family of skew Hadamard difference sets, {\it J. Combin. Theory} (A), {\bf 113} (2006), 1526--1535.

\bibitem{dwx} C. Ding, Z. Wang and Q. Xiang,  Skew Hadamard difference sets from the Ree-Tits slice symplectic spreads in ${\rm PG}(3,3^{2h+1})$,  {\it J. Combin. Theory} (A), {\bf 114 } (2007), 867--887.


\bibitem{evans} R. J. Evans, Nonexistence of twentieth power residue difference sets, {\it Acta Arith.} {\bf 84} (1999), 397-402. 

\bibitem{tfeng} T. Feng, Non-abelian skew Hadamard difference sets fixed by a prescribed automorphism,  {\it J. Combin. Theory} (A), {\bf 118 } (2011), 27--36.

\bibitem{FX} T. Feng, Q. Xiang, Strongly regular graphs from union of cyclotomic classes, {\tt ArXiv: 1010.4107v2}.

\bibitem{ikutam} T. Ikuta, A. Munemasa, Pseudocyclic association schemes and strongly regular graphs, {\it Europ. J. Combin.} {\bf 31} (2010), 1513--1519. 


\bibitem{kantor} W. M. Kantor, 2-Transitive symmetric designs, {\it Trans. Amer. Math. Soc.}, {\bf 146} (1969), 1--28.

\bibitem{kibler} R. E. Kibler, A summary of noncyclic difference sets, $k<20$, {\it J. Combin. Theory} (A), {\bf 25} (1978), 62--67.

\bibitem{Lander} E. S. Lander, {\it Symmetric designs: an algebraic approach}, Cambridge University Press, 1983.


\bibitem{del} C. L. M.  de Lange, Some new cyclotomic strongly regular graphs, {\it J. Alg. Combin.}  {\bf 4} (1995), 329--330.

\bibitem{Lang} P. Langevin, Calculs de certaines sommes de Gauss, {\it J. Number Theory}, {\bf 63} (1997), 59--64. 


\bibitem{Mbo} O. D. Mbodj, Quadratic Gauss sums, {\it Finite Fields and Appl.}, {\bf 4} (1998), 347--361.

\bibitem{McE} R. J. McEliece, Irreducible cyclic codes and Gauss sums. Combinatorics (Proc. NATO Advanced Study Inst., Breukelen, 1974), Part 1: Theory of designs, finite geometry and coding theory, pp. 179--196. Math. Centre Tracts, No. 55, Math. Centrum, Amsterdam.

\bibitem{mv} P. Meijer, M. van der Vlugt, The evaluation of Gauss sums for characters of 2-power order, {\it J Number Theory}, {\bf 100} (2003), 381--395.

\bibitem{muzy} M. E. Muzychuk, On skew Hadamard difference sets, {\tt arXiv:1012.2089}.

\bibitem{paley} R. E. A. C. Paley, On orthogonal matrices, {\it J. Math. Phys.} {\bf 12} (1933), 311--320.

\bibitem{storer} T. Storer, {\it Cyclotomy and difference sets}, Markham, Chicago, 1967.

\bibitem{wh} G. B. Weng, L. Hu, Some results on skew Hadamard difference sets, {\it Des. Codes and Cryptogr.} {\bf 50} (2009), 93--105.


\bibitem{wqwx} G. B. Weng, W. S. Qiu, Z. Wang and Q. Xiang,  Pseudo-Paley graphs and skew Hadamard difference sets from presemifields, {\it Des. Codes and Cryptogr.} {\bf 44} (2007), 49--62.

\bibitem{Xiang} Q. Xiang, Recent progress in algebraic design theory, {\it Finite Fields and Their Aplications} (Ten Year Anniversary Edition) {\bf 11} (2005), 622--653.

\bibitem{yx} J. Yang, L. Xia, Complete solving of explicit evaluation of Gauss sums in the index 2 case, {\it  Sci. China Ser. A} {\bf 53} (2010), 2525--2542.

\bibitem{yx2} J. Yang, L. Xia, A note on the sign (unit root) ambiguities of Gauss sums in index $2$ and $4$ case, arXiv:0912.1414v1.

\end{thebibliography}
\end{document}